\documentclass{amsart}
\usepackage{amssymb}
\usepackage[abbrev]{amsrefs}
\numberwithin{equation}{section}
\def\N{\mathbb N}
\def\R{\mathbb R}

\providecommand{\norm}[1]{\left\lVert#1\right\rVert}
\newcommand{\ip}[2]{\left\langle #1, #2 \right\rangle}

\providecommand{\abs}[1]{\left\lvert#1\right\rvert}
\DeclareMathOperator*{\Argmin}{argmin}
\DeclareMathOperator{\Fix}{\mathcal{F}}
\DeclareMathOperator{\AC}{\mathcal{A}}

\DeclareMathOperator{\Prox}{Prox}
\newcommand{\CAT}{\textup{CAT}}

\theoremstyle{plain}
\newtheorem{theorem}{Theorem}[section]
\newtheorem{lemma}[theorem]{Lemma}

\newtheorem{corollary}[theorem]{Corollary}
\newtheorem{example}[theorem]{Example}
\theoremstyle{definition}

\theoremstyle{remark}
\newtheorem{remark}[theorem]{Remark}
\allowdisplaybreaks
\title[Metrically nonspreading mappings in Hadamard spaces] 
{Fixed points of metrically nonspreading mappings in Hadamard spaces}
\author[F.~Kohsaka]{Fumiaki~Kohsaka}
\address[F.~Kohsaka]
{Department of Mathematical Sciences, Tokai University, 
Kitakaname, Hiratsuka, Kanagawa 259-1292, Japan}
\email{f-kohsaka@tsc.u-tokai.ac.jp}
\begin{document}
\begin{abstract}
 We study the existence and approximation of fixed points of 
 metrically nonspreading mappings and 
 firmly metrically nonspreading mappings in Hadamard spaces. 
 The resolvents of monotone operators satisfying range conditions 
 are typical examples of firmly metrically nonspreading mappings. 
 Applications to monotone operators in such spaces are also included. 
\end{abstract}
\maketitle

\section{Introduction}
\label{sec:intro}

A number of nonlinear variational problems can be formulated as  
the problem of finding zero points of 
maximal monotone operators in Banach spaces. 
Among those problems are 
convex minimization problems~\cites{MR0193549, MR0262827}, 
variational inequality problems~\cite{MR0282272}, 
saddle point problems~\cite{MR0285942}, 
and equilibrium problems~\cite{MR2422998}. 

The class of nonspreading mappings first introduced by 
Kohsaka and Takahashi~\cite{MR2430800} 
is closely related to the problem of finding zero points of 
maximal monotone operators in Banach spaces. 
The authors of the papers~\cites{MR2430800, MR2448915, MR3289781} 
obtained some basic results on the fixed point problem 
for nonspreading mappings and applied them 
to maximal monotone operators in Banach spaces. 

Recall that a mapping $T$ of a nonempty 
subset $C$ of a smooth real Banach space $E$ 
into itself is said to be nonspreading if 
 \[
  \phi (Tx, Ty) + \phi (Ty, Tx) \leq \phi (Tx, y) + \phi (Ty, x)
 \]
for all $x,y\in C$, where $\phi$ is the two variable 
real function~\cites{MR1386667, MR1972223, MR1274188} 
on $E^2$ defined by 
\[
\phi (u, v) = \norm{u}^2 - 2\ip{u}{Jv} + \norm{v}^2 
\]
for all $u, v\in E$ and $J$ denotes the normalized duality mapping 
of $E$ into $E^*$. It is known~\cites{MR2430800, MR2448915} that 
if $E$ is a smooth, strictly convex, 
and reflexive Banach space, then the following hold. 
\begin{itemize}
 \item The generalized projection $\mathit{\Pi}_C$~\cites{MR1386667, MR1274188} 
of $E$ onto a nonempty closed convex subset $C$ of $E$ is 
nonspreading and $\Fix(\mathit{\Pi}_C)=C$; 
 \item the resolvent $Q_A$~\cites{MR2144037, MR2112848, MR2058504}  
of a maximal monotone operator 
$A\colon E\to 2^{E^*}$ defined by 
$Q_A=(J+A)^{-1}J$ is nonspreading and 
$\Fix(Q_A)=A^{-1}(0)$. 
\end{itemize}

We know the following results for nonspreading mappings in Banach spaces. 

\begin{theorem}[\cite{MR2430800}*{Theorem~4.1}] 
 \label{thm:fpt-nonsp-B}
 Let $C$ be a nonempty closed convex subset of a smooth, strictly
 convex, and reflexive real Banach space $E$ 
 and $T$ a nonspreading mapping of $C$ into itself. 
 Then $\Fix(T)$ is nonempty if and only if 
 $\{T^nx\}$ is bounded for some $x\in C$. 
\end{theorem}

\begin{theorem}[\cite{MR2430800}*{Proposition~3.2}] 
 \label{thm:demi-nonsp-B}
 Let $C$ be a nonempty closed convex subset of a strictly convex 
 real Banach space $E$ with a uniformly G\^ateaux differentiable norm, 
 $T$ a nonspreading mapping of $C$ into itself, 
 and $u$ an element of $C$ such that 
 there exists a sequence $\{x_n\}$ in $C$ which is 
 weakly convergent to $u$ and satisfies 
 $\norm{x_n-Tx_n}\to 0$ as $n\to \infty$. 
 Then $u$ is an element of $\Fix(T)$. 
\end{theorem}

\begin{theorem}[\cite{MR2430800}*{Theorem~4.6}] 
 \label{thm:fpt-commute-nonsp-B}
 Let $C$ be a nonempty bounded closed convex subset of a smooth, strictly
 convex, and reflexive real Banach space $E$ 
 and $\{T_{\alpha}\}_{\alpha\in A}$ 
 a commutative family of nonspreading mappings of $C$ into itself. 
 Then $\bigcap_{\alpha \in A}\Fix(T_{\alpha})$ is nonempty. 
\end{theorem}

In particular, 
we know the following result on the asymptotic behavior of 
the Mann iteration~\cite{MR0054846} for nonspreading mappings in Banach spaces. 

\begin{theorem}[\cite{MR3289781}*{Theorem~4.1}]
 \label{thm:Mann-nonsp-B}
 Let $C$ be a nonempty closed convex subset of a smooth, strictly convex,
 and reflexive real Banach space $E$, 
 $J$ the normalized duality mapping of $E$ into $E^*$, 
 $\mathit{\Pi}_C$ the generalized projection of $E$ onto $C$, 
 $T$ a nonspreading 
 mapping of $C$ into itself, $\{\alpha_n\}$ a sequence in $(0,1]$ such that 
 $\sum_{n=1}^{\infty} \alpha_n = \infty$, and 
 both $\{x_n\}$ and $\{z_n\}$ sequences in $C$ defined by 
 $x_1\in C$ and 
 \begin{align*}
  x_{n+1} &= \mathit{\Pi}_C J^{-1} \bigl((1-\alpha_n) Jx_n 
 + \alpha_n JTx_n \bigr); \\
  z_n &= \frac{1}{\sum_{l=1}^{n}\alpha_l} 
  \sum_{k=1}^{n} \alpha_k Tx_k
 \end{align*}
 for all $n\in \N$. 
 Then the following are equivalent. 
 \begin{enumerate}
  \item[(i)] $\Fix(T)$ is nonempty; 
  \item[(ii)] $\{x_n\}$ is bounded; 
  \item[(iii)] $\{z_n\}$ is bounded; 
  \item[(iv)] $\{z_n\}$ has a bounded subsequence. 
 \end{enumerate}
 In this case, each subsequential weak limit of $\{z_n\}$ 
 belongs to $\Fix(T)$. 
\end{theorem}

In the special case where $E$ is a real Hilbert space, 
a mapping $T$ of a nonempty subset $C$ of $E$ into itself is nonspreading if 
\begin{align}\label{eq:nonsp-Hilbert}
 2\norm{Tx-Ty}^2 \leq \norm{Tx-y}^2 + \norm{Ty-x}^2
\end{align}
for all $x,y\in C$. This condition is satisfied 
whenever $T$ is firmly nonexpansive, i.e., 
\begin{align}\label{eq:fn-Hilbert}
 \norm{Tx-Ty}^2 \leq \ip{Tx-Ty}{x-y}
\end{align}
for all $x,y\in C$. 

Motivated by~\eqref{eq:nonsp-Hilbert} and~\eqref{eq:fn-Hilbert},  
we say that a mapping $T$ of a metric space $(X, d)$ 
into itself is 
\begin{itemize}
 \item metrically nonspreading if 
 \[
  2d(Tx,Ty)^2 \leq d(Tx,y)^2 + d(Ty,x)^2 
 \]
 for all $x,y\in X$; 
 \item firmly metrically nonspreading if 
 \[
  2d(Tx,Ty)^2 + d(Tx,x)^2 +d(Ty,y)^2 \leq d(Tx,y)^2 + d(Ty,x)^2 
 \]
 for all $x,y\in X$. 
\end{itemize}

Every firmly metrically nonspreading 
mapping is obviously metrically nonspreading. 
We can also see that every metrically nonspreading mapping 
with a fixed point is quasinonexpansive 
and that every firmly metrically nonspreading mapping $T$ 
of $X$ into itself with a fixed point satisfies 
\begin{align}\label{eq:fmn-quasi}
 d(u, Tx)^2 + d(Tx, x)^2 \leq d(u, x)^2
\end{align}
for all $u\in \Fix(T)$ and $x\in X$. 
If $X$ is a nonempty subset of a real Hilbert space, then 
\[
 \ip{x-y}{z-w} = \frac{1}{2}
 \bigl(\norm{x-w}^2 + \norm{y-z}^2 - \norm{x-z}^2 - \norm{y-w}^2\bigr)
\] 
for all $x,y,z,w\in X$, where the left hand side is the inner product on the space. 
In this case, 
$T$ is firmly metrically nonspreading 
if and only if it is firmly nonexpansive. 

Metrically nonspreading mappings and firmly metrically nonspreading 
mappings are also called $1/2$-nonexpansive mappings 
and firmly nonexpansive mappings by 
Naraghirad, Wong, and Yao~\cite{MR3037927}*{Definition~4.6} 
and Khatibzadeh and Ranjbar~\cite{MR3679017}*{Definition~3.5}, respectively. 
The notion of $\alpha$-nonexpansive mapping 
was first introduced by Aoyama and Kohsaka~\cite{MR2810735}*{Definition~2.2} 
in the context of Banach spaces. 

As we see in Sections~\ref{sec:example-mn} and~\ref{sec:app}, 
the metric projections onto nonempty closed convex sets, 
the proximity mappings of proper lower semicontinuous convex functions, 
and the resolvents of monotone operators satisfying range conditions  
in Hadamard spaces are firmly metrically nonspreading. 
Thus the fixed point problem for such mappings 
is closely related to convex analysis in Hadamard spaces. 

The aim of this paper is to study the existence and approximation of 
fixed points of metrically nonspreading mappings and 
firmly metrically nonspreading mappings in Hadamard spaces. 
In particular, we obtain analogues of 
Theorems~\ref{thm:fpt-nonsp-B},~\ref{thm:demi-nonsp-B},~\ref{thm:fpt-commute-nonsp-B}, and~\ref{thm:Mann-nonsp-B} 
for metrically nonspreading mappings in Hadamard spaces. 
We finally apply our results to monotone operators in Hadamard spaces. 

\section{Preliminaries}
\label{sec:pre}

Throughout this paper, we denote by $\N$ and $\R$ the sets 
of all positive integers and real numbers, respectively. 
The two dimensional Euclidean space and its norm 
are denoted by $\R^2$ and $\abs{\,\cdot\,}_{\R^2}$, 
respectively.  
Unless otherwise specified, 
we denote by $X$ a metric space with a metric $d$. 
The set of all fixed points of a mapping $T$ of $X$ into itself 
is denoted by $\Fix(T)$. 
A mapping $T$ of $X$ into itself is said to be
\begin{itemize}
 \item asymptotically regular if $\lim_{n} d(T^{n+1}x, T^{n}x) = 0$ 
 for all $x\in X$; 
 \item nonexpansive if $d(Tx, Ty)\leq d(x, y)$ for all $x,y\in X$; 
 \item quasinonexpansive if $\Fix(T)$ is nonempty 
and $d(u, Tx)\leq d(u, x)$ for all $u\in \Fix(T)$ and $x\in X$. 
\end{itemize}

The product space $X\times X$ 
and its element $(x,y)$ are denoted 
by $X^2$ and $\overrightarrow{xy}$, respectively. 
The quasilinearization 
$\ip{\cdot}{\cdot}$ on $X^2$ introduced by 
Berg and Nikolaev~\cite{MR2390077} 
is a real function on $X^2\times X^2$ defined by 
\begin{align}\label{eq:quasilinearization}
 \ip{\overrightarrow{xy}}{\overrightarrow{zw}} 
 = \frac{1}{2} \left(
 d(x,w)^2 + d(y, z)^2 - d(x,z)^2 - d(y, w)^2 
 \right)
\end{align}
for all $\overrightarrow{xy}, \overrightarrow{zw}\in X^2$. 
If $X$ is particularly a real Hilbert space, then 
\[
  \ip{\overrightarrow{xy}}{\overrightarrow{zw}} 
 = \ip{x-y}{z-w}
\]
for all $x,y,z,w\in X$. 
It is clear that 
\begin{itemize}
 \item $\ip{\overrightarrow{xy}}{\overrightarrow{xy}}=d(x,y)^2$; 
 \item $\ip{\overrightarrow{xy}}{\overrightarrow{zw}}
 =\ip{\overrightarrow{zw}}{\overrightarrow{xy}} 
 =-\ip{\overrightarrow{yx}}{\overrightarrow{zw}}$;
 \item $\ip{\overrightarrow{xp}}{\overrightarrow{zw}}
+ \ip{\overrightarrow{py}}{\overrightarrow{zw}}
 =\ip{\overrightarrow{xy}}{\overrightarrow{zw}}$; 
 \item $d(x,y)^2 = d(x,z)^2 + d(z,y)^2 
 + 2\ip{\overrightarrow{xz}}{\overrightarrow{zy}}$
\end{itemize}
for all $x,y,z,w,p\in X$. 

A metric space $X$ is said to be uniquely geodesic 
if for each $x,y\in X$, there exists 
a unique mapping $\gamma$ of $[0,l]$ into $X$ 
such that $\gamma(0)=x$, $\gamma(l)=y$, 
and 
\[
 d\bigl(\gamma(s), \gamma(t)\bigr) = \abs{s-t}
\]
for all $s, t \in [0,l]$, where $l=d(x,y)$. 
The mapping $\gamma$ is called a geodesic from $x$ to $y$ 
and the point $\gamma(\alpha l)$ is denoted 
by $(1-\alpha)x \oplus \alpha y$ for all $\alpha \in [0,1]$. 
A metric space $X$ is said to be a $\CAT(0)$ space 
if it is uniquely geodesic and the following 
$\CAT(0)$ inequality 
\[
 d\bigl((1-\alpha) x \oplus \alpha y, 
 (1-\beta) x \oplus \beta z \bigr) 
 \leq \abs{(1-\alpha) \bar{x} + \alpha \bar{y} 
-\bigl((1-\beta) \bar{x} + \beta \bar{z} \bigr)}_{\R^2}
\]
holds whenever $x,y,z\in X$, $\bar{x}, \bar{y}, \bar{z}\in \R^2$, 
\[
 d(x,y) = \abs{\bar{x}-\bar{y}}_{\R^2}, \quad 
 d(y,z) = \abs{\bar{y}-\bar{z}}_{\R^2}, \quad 
 d(z,x) = \abs{\bar{z}-\bar{x}}_{\R^2}, 
\]
and $\alpha, \beta\in [0,1]$. 
It follows from~\cite{MR2390077}*{Corollary~3} 
and~\cite{MR3241330}*{Theorem~1.3.3~(v)} that 
a uniquely geodesic metric space $X$ is a $\CAT(0)$ space if and only if 
the following Cauchy--Schwarz inequality 
\begin{align}\label{eq:CS-ineq}
 \abs{\ip{\overrightarrow{xy}}{\overrightarrow{zw}}} 
 \leq d(x,y)d(z,w)
\end{align}
holds for all $\overrightarrow{xy}, \overrightarrow{zw}\in X^2$. 

It is obvious that if $X$ is a $\CAT(0)$ space, then 
\begin{itemize}
 \item $d(z,  (1-\alpha) x \oplus \alpha y) 
 \leq (1-\alpha) d(z, x) + \alpha d(z, y)$; 
 \item $d(z, (1-\alpha) x \oplus \alpha y)^2 
 \leq (1-\alpha) d(z, x)^2 + \alpha d(z, y)^2 
 -\alpha (1-\alpha) d(x,y)^2$
\end{itemize} 
for all $x,y,z\in X$ and $\alpha \in [0,1]$.  
A complete $\CAT(0)$ space is called an Hadamard space. 
Among typical examples of Hadamard spaces are 
nonempty closed convex subsets of real Hilbert spaces, 
open unit balls of complex Hilbert spaces with hyperbolic metrics, 
and simply connected complete Riemannian manifolds  
with nonpositive sectional curvature; 
see~\cites{MR3241330, MR1744486, MR1835418} 
on geodesic metric spaces and $\CAT(0)$ spaces for more details.  

It is well known that if $X$ is a $\CAT(0)$ space 
and $T$ is a quasinonexpansive mapping of $X$ into itself, 
then $\Fix(T)$ is closed and convex. 
Hence the fixed point set of every metrically nonspreading
mapping with a fixed point is closed and convex. 

The concept of $\Delta$-convergence, first introduced by 
Lim~\cite{MR423139} and later applied to the study of 
$\CAT(0)$ spaces by Kirk and Panyanak~\cite{MR2416076}, 
is a generalization of 
weak convergence in the context of Hilbert spaces 
to metric spaces. 
The asymptotic center $\AC\bigl(\{x_n\}\bigr)$ 
of a sequence $\{x_n\}$ in a metric space $X$ is defined by 
\[
 \AC\bigl(\{x_n\}\bigr) 
 = \left\{u\in X: 
 \limsup_{n} d(u, x_n) = \inf_{y\in X} 
 \limsup_{n} d(y, x_n)\right\}, 
\]
which coincides with the whole space $X$ if $\{x_n\}$ is unbounded. 
The sequence $\{x_n\}$ is said to be $\Delta$-convergent to 
$p\in X$ if 
\[
 \AC\bigl(\{x_{n_i}\}\bigr) = \{p\}
\]
for each subsequence $\{x_{n_i}\}$ of $\{x_n\}$, 
in which case $p$ is said to be the $\Delta$-limit of $\{x_n\}$. 
We denote by $\omega_{\Delta}\bigl(\{x_n\}\bigr)$ 
the set of all subsequential $\Delta$-limits 
of $\{x_n\}$. It is obvious that if $\{x_n\}$ is $\Delta$-convergent 
to $p$, then $\{x_n\}$ is bounded and 
$\omega_{\Delta}\bigl(\{x_n\}\bigr)=\{p\}$. 

The following lemmas are of fundamental importance. 

\begin{lemma}[\cite{MR2232680}*{Proposition~7}; 
see also~\cite{MR3241330}*{Section~3.1}]\label{lem:AC}
 The asymptotic center of every bounded sequence in an Hadamard space 
 is a singleton. 
\end{lemma}

\begin{lemma}[\cite{MR2416076}*{Section~3}; 
see also~\cite{MR3241330}*{Proposition~3.1.2}]\label{lem:Delta-conv}
 Every bounded sequence in an Hadamard space 
 has a $\Delta$-convergent subsequence. 
\end{lemma}

\begin{lemma}[\cite{MR3574140}*{Lemma~2.6}; 
see also~\cite{MR3213144}*{Proposition~3.1}]\label{lem:KK-Delta}
 If $\{x_n\}$ is a bounded sequence 
 in an Hadamard space such that $\{d(z, x_n)\}$ is convergent 
 for each $z$ in $\omega_{\Delta}\bigl(\{x_n\}\bigr)$, 
 then $\{x_n\}$ is $\Delta$-convergent. 
\end{lemma}

A subset $C$ of a $\CAT(0)$ space $X$ is said to be 
convex if $(1-\alpha) x \oplus \alpha y \in C$ whenever 
$x,y\in C$ and $\alpha \in [0,1]$. 
A function $f$ of $X$ into $(-\infty, \infty]$ is said to be 
proper if $f(p)$ is finite for some $p \in X$. 
It is also said to be convex if 
\[
 f((1-\alpha) x \oplus \alpha y) \leq (1-\alpha) f(x) + \alpha f(y)
\]
whenever $x,y\in X$ and $\alpha \in (0,1)$. 
The set of all minimizers of a function $f$ of $X$ into $(-\infty, \infty]$ 
is denoted by $\Argmin_X f$ or $\Argmin_{y\in X} f(y)$. 
If $\Argmin_X f$ is a singleton $\{p\}$ for some $p\in X$, 
we sometimes identify $\Argmin_X f$ with $p$.   

We know the following minimization theorem in Hadamard spaces. 

\begin{theorem}[\cite{MR3574140}*{Theorem~4.1}]
\label{thm:unique_minimizer}
 Let $\{z_n\}$ be a bounded sequence in 
 an Hadamard space $X$, 
 $\{\beta_n\}$ a sequence of positive real numbers such that 
 $\sum_{n=1}^{\infty}\beta_n=\infty$, 
 and $g$ the real function on $X$ defined by 
 \begin{align*}
  g(y) = \limsup_{n} \frac{1}{\sum_{l=1}^{n}\beta_l} 
 \sum_{k=1}^{n} \beta_k d(y, z_k)^2
 \end{align*}
 for all $y\in X$. 
 Then $g$ is a continuous and convex 
 function such that $\Argmin_X g$ is a singleton. 
\end{theorem}

We also know the following lemmas. 

\begin{lemma}[\cite{MR3213164}*{Lemma~11}]\label{lem:double-sequence}
Let $A$ be a bounded function 
of $\N \times \N$ into $[0,\infty)$ such that 
$A(n,n)=0$ for all $n\in \N$ and 
\[
 2A(n+1,m+1) \leq A(n+1,m) + A(n,m+1)
\]
for all $n, m \in \N$. 
Then $\lim_{n} A(n,n+1) = 0$.  
\end{lemma}

\begin{lemma}[\cite{MR3777000}*{Lemma~2.5}]\label{lem:limsup}
 Let $I$ be a nonempty closed subset of $\R$, 
 $\{t_n\}$ a bounded sequence in $I$, 
 and $f$ a nondecreasing and continuous real function on $I$. 
 Then 
 \[
  f\left(\limsup_n t_n\right)= \limsup_n f(t_n).  
 \]
\end{lemma}

\section{Examples of metrically nonspreading mappings}
\label{sec:example-mn}

In this section, we discuss some examples of metrically nonspreading 
mappings and firmly metrically nonspreading mappings in Hadamard spaces. 

Using~\eqref{eq:quasilinearization} and~\eqref{eq:CS-ineq}, 
we readily obtain the following. 
We note that~\eqref{eq:fn-metric} is equivalent to~\eqref{eq:fn-Hilbert} 
when $X$ is a nonempty subset of a real Hilbert space. 
\begin{lemma}\label{lem:fund-fmn-mn}
Let $X$ be a metric space and $T$ a mapping of $X$ into itself. 
Then $T$ is firmly metrically nonspreading if and only if 
\begin{align}\label{eq:fn-metric}
 d(Tx,Ty)^2 \leq \ip{\overrightarrow{(Tx)(Ty)}}{\overrightarrow{xy}}
\end{align}
for all $x,y\in X$. 
If $X$ is a $\CAT(0)$ space and $T$ is firmly metrically nonspreading, 
then $T$ is nonexpansive. 
\end{lemma}

The metric projections and the proximity mappings in Hadamard spaces 
are two particularly important examples 
of firmly metrically nonspreading mappings. 

Let $X$ be an Hadamard space. 
If $C$ is a nonempty closed convex subset of $X$, 
then the metric projection $P_C$ of $X$ onto $C$ given by 
\[
 P_C(x) = \Argmin_{y\in C} d(y, x)
\]
for all $x\in X$ is a well-defined nonexpansive mapping 
of $X$ onto $C$ and $\Fix(P_C)=C$; 
see~\cites{MR1744486, MR3241330} for more details.  
More generally, if $f$ is a proper lower semicontinuous convex 
function of $X$ into $(-\infty, \infty]$, then 
the proximity mapping $\Prox_f$ of $f$ given by 
\begin{align}\label{eq:prox}
  \Prox_f (x) = \Argmin_{y\in X} \left\{f(y) + \frac{1}{2}d(y, x)^2\right\}
\end{align}
for all $x\in X$ is a well-defined nonexpansive mapping 
of $X$ into itself and $\Fix(\Prox_f)=\Argmin_X f$;   
see~\cites{MR3241330, MR1360608, MR1651416}. 
It is also known~\cite{MR3206460}*{Proposition~3.3} that 
$\Prox_f$ is firmly nonexpansive, i.e., 
\begin{align*}
 d(\Prox_f x, \Prox_f y) 
 \leq 
 d\bigl(\alpha x \oplus (1-\alpha) \Prox_f x, 
 \alpha y \oplus (1-\alpha) \Prox_f y \bigr)
\end{align*}
for all $\alpha \in [0,1]$ and $x,y\in X$. 
The proximity mappings in Hadamard spaces were originally studied by 
Jost~\cite{MR1360608} and Mayer~\cite{MR1651416}; 
see also~\cites{MR3047087, MR3691338, MR3396547, MR3574140} on 
minimization algorithms based on the proximity mappings in Hadamard spaces. 
It follows from~\cite{MR3574140}*{Corollary~3.2} 
and Lemma~\ref{lem:mono-res} that the following holds. 

\begin{example}\label{expl:proj-prox}
 If $X$ is an Hadamard space, then the following hold. 
 \begin{enumerate}
  \item[(i)] The metric projection $P_C$ of $X$ onto a nonempty 
 closed convex subset $C$ is firmly metrically nonspreading; 
  \item[(ii)] the proximity mapping $\Prox_f$ of a proper 
 lower semicontinuous convex function $f$ 
 of $X$ into $(-\infty, \infty]$ is firmly metrically nonspreading. 
 \end{enumerate}
\end{example}

Motivated by~\cite{MR2810735}*{Example~2.4}, we show the following result. 
\begin{example}\label{expl:mn}
 Let $X$ be a metric space, both $S$ and $T$ metrically nonspreading mappings 
 such that $S(X)$ and $T(X)$ are contained by a closed ball $\overline{B}_{r}(a)$ 
 for some $a\in X$ and $r>0$, $\delta$ a positive real number satisfying 
 $\delta\geq \left(1+2\sqrt{2}\right)r$, 
 and 
 $U$ the mapping of $X$ into itself defined by 
 \begin{align*}
  Ux=
  \begin{cases}
  Sx & (x\in \overline{B}_{\delta}(a)); \\
  Tx & (\textrm{otherwise}). 
  \end{cases}
 \end{align*}
 Then $U$ is metrically nonspreading. 
\end{example}

\begin{proof}
 Let $x,y\in X$ be given. 
 If either $x,y\in \overline{B}_{\delta}(a)$ or 
 $x,y\in X\setminus \overline{B}_{\delta}(a)$, then we have 
 \[
  2d(Ux,Uy)^2 \leq d(Ux, y)^2 + d(Uy, x)^2 
 \]
 since both $S$ and $T$ are metrically nonspreading. 
 If $x\in \overline{B}_{\delta}(a)$ 
 and $y\in X\setminus \overline{B}_{\delta}(a)$, 
 then we have 
 \begin{align*}
  &d(Ux, y)^2 + d(Uy, x)^2 \\
  &\geq d(Ux, y)^2 = d(Sx, y)^2 
  \geq \bigl(d(y, a) - d(Sx, a)\bigr)^2 > (\delta - r\bigr)^2 \geq 8r^2
 \end{align*}
 and 
 \[
  8r^2 
  \geq 2\bigl(d(Sx, a) + d(a, Ty)\bigr)^2
  \geq 2d(Sx, Ty)^2 = 2d(Ux, Uy)^2. 
 \]
 Hence we have 
 \begin{align}\label{eq:expl:mn-a}
  d(Ux, y)^2 + d(Uy, x)^2 > 2d(Ux, Uy)^2. 
 \end{align}
 If $x\in X\setminus \overline{B}_{\delta}(a)$ 
 and $y\in \overline{B}_{\delta}(a)$, then we also obtain~\eqref{eq:expl:mn-a}. 
 Therefore, the mapping $U$ is metrically nonspreading. 
\end{proof}

\begin{remark}
 It follows from Lemma~\ref{lem:fund-fmn-mn} and Example~\ref{expl:mn} that 
 there exists a metrically nonspreading mapping 
 which is not firmly metrically nonspreading. 
 In fact, let $X$ be an unbounded Hadamard space. 
 Then we have $a\in X$ and $\delta>0$ 
 such that $X\setminus \overline{B}_{\delta}(a)$ 
 is nonempty. 
 Let $r$ be the same as in Example~\ref{expl:mn}, 
 both $S$ and $T$ the metric projections of $X$ onto 
 $\{a\}$ and the closed ball $\overline{B}_r(a)$, respectively, 
 and $U$ the mapping defined as in Example~\ref{expl:mn}. 
 Then it follows from Example~\ref{expl:mn} 
 that $U$ is metrically nonspreading. 
 On the other hand, $U$ is discontinuous 
 at any $x\in X$ with $d(x, a)=\delta$.  
 Lemma~\ref{lem:fund-fmn-mn} implies that 
 $U$ is not firmly metrically nonspreading. 
\end{remark}

\section{Fixed points of metrically nonspreading mappings}
\label{sec:Fixed}

In this section, we study some fundamental properties of metrically
nonspreading mappings in Hadamard spaces. 

\begin{lemma}\label{lem:fpt}
 Let $X$ be a metric space, $T$ a metrically nonspreading mapping 
 of $X$ into itself, and 
 $\{x_n\}$ a sequence in $X$ such that 
 $\AC\bigl(\{x_n\}\bigr)=\{p\}$ for some $p\in X$. 
 If 
 \begin{itemize}
  \item $\limsup_n d(x_n, p) = \limsup_n d(Tx_n, p)$; 
  \item $\limsup_n d(x_n, Tp) = \limsup_n d(Tx_n, Tp)$, 
 \end{itemize}
 then $p$ is an element of $\Fix(T)$. 
\end{lemma}

\begin{proof}
 Since $\AC\bigl(\{x_n\}\bigr) = \{p\}$, 
 the sequence $\{x_n\}$ is bounded. 
 In fact, if $X$ is a singleton, then $\{x_n\}$ is obviously bounded.  
 In the other case, we have $q\in X$ which 
 is distinct from $p$ and hence 
 \[
  \limsup_{n} d(x_n, p) < \limsup_{n} d(x_n, q). 
 \]
 This implies the boundedness of $\{x_n\}$. 

 On the other hand, since $T$ is metrically nonspreading, we have 
 \[
  2d(Tx_n, Tp)^2 \leq d(Tx_n, p)^2 + d(Tp, x_n)^2. 
 \]
 Taking the upper limit gives us that 
 \[
  2\limsup_{n} d(Tx_n, Tp)^2 
  \leq \limsup_{n} d(Tx_n, p)^2 + \limsup_{n} d(Tp, x_n)^2. 
 \]
 By assumptions, we have 
 \begin{align}\label{eq:lem:fpt-a}
  2\limsup_{n} d(x_n, Tp)^2 
  \leq \limsup_{n} d(x_n, p)^2 + \limsup_{n} d(x_n, Tp)^2. 
 \end{align}
 Since $\{d(x_n, Tp)^2\}$ is a bounded sequence in $[0,\infty)$, 
 it follows from~\eqref{eq:lem:fpt-a} 
 and Lemma~\ref{lem:limsup} that 
 \begin{align*}
  \left(\limsup_{n} d(x_n, Tp)\right)^2 
  &=\limsup_{n} d(x_n, Tp)^2 \\
  &\leq \limsup_{n} d(x_n, p)^2 
  =\left(\limsup_{n} d(x_n, p)\right)^2 
 \end{align*}
 and hence 
 \[
  \limsup_{n}d(x_n, Tp) \leq \limsup_{n} d(x_n, p). 
 \]
 It then follows from $\AC\bigl(\{x_n\}\bigr)=\{p\}$ that $Tp=p$.  
\end{proof}

Using Lemma~\ref{lem:fpt}, we first obtain 
the following fixed point theorem for metrically nonspreading 
mappings in Hadamard spaces. 
This result also follows from the result~\cite{MR3037927}*{Lemma~4.7}. 
We note that the proof of~\cite{MR3037927}*{Lemma~4.7} 
is valid to the case where $1-2\alpha \geq 0$. 

\begin{theorem}\label{thm:fpt}
 Let $X$ be an Hadamard space 
 and $T$ a metrically nonspreading mapping of $X$ into itself. 
 Then $\Fix(T)$ is nonempty if and only if 
 $\{T^nx\}$ is bounded for some $x\in X$. 
\end{theorem}

\begin{proof}
 Since the only if part is obvious, 
 it is sufficient to prove the if part. 
 Suppose that $\{T^nx\}$ is bounded for some $x\in X$ 
 and let $\{x_n\}$ be the sequence in $X$ defined by 
 $x_n=T^nx$ for all $n\in \N$. 
 It then follows from Lemma~\ref{lem:AC} that 
 $\AC\bigl(\{x_n\}\bigr) =\{p\}$ for some $p\in X$.  
 By the definition of $\{x_n\}$, we have 
 \[
  \limsup_{n} d(x_n, y) = \limsup_{n} d(Tx_n, y)
 \]
 for all $y\in X$. Thus it follows from Lemma~\ref{lem:fpt} 
 that $p$ is an element of $\Fix(T)$. 
\end{proof}

As a direct consequence of Theorem~\ref{thm:fpt}, 
we obtain the following corollary. 

\begin{corollary}\label{cor:fpt}
 Every metrically nonspreading mapping of a bounded Hadamard space 
 into itself has a fixed point. 
\end{corollary}

We can also show the following common fixed point theorem. 
\begin{theorem}
 \label{prob:thm-finite-commut-nonsp}
 Let $X$ be a bounded Hadamard space 
 and $\{T_{k}\}_{k=1}^{m}$ 
 a commutative finite family of metrically nonspreading mappings 
 of $X$ into itself. 
 Then $\bigcap_{k=1}^{m}\Fix(T_{k})$ is nonempty. 
\end{theorem}

\begin{proof}
 The proof is given by induction on $m$. 
 Theorem~\ref{thm:fpt} implies that 
 the conclusion holds if $m=1$. 
 Suppose that the conclusion holds for some $m=l \in \N$ 
 and let $\{T_{k}\}_{k=1}^{l+1}$ be a commutative family 
 of metrically nonspreading mappings of $X$ into itself. 
 Then the set $Y$ given by $Y=\bigcap_{k=1}^{l}\Fix(T_k)$ 
 is a nonempty closed convex subset of $X$. 
 Hence $Y$ is an Hadamard space.    
 We next show that $T_{l+1}Y$ is contained by $Y$. 
 In fact, if $v\in Y$ and $k\in \{1,2,\dots ,l\}$, 
 then it follows from $T_kv=v$ and $T_{l+1}T_{k}=T_{k}T_{l+1}$ 
 that 
 \[
  T_{l+1}v = T_{l+1} T_{k} v = T_{k} T_{l+1} v
 \]
 and hence $T_{l+1}v\in \Fix(T_{k})$. 
 Thus $T_{l+1}Y$ is a subset of $Y$. 
 Accordingly, the restriction of $T_{l+1}$ to 
 the Hadamard space $Y$ is a metrically nonspreading self mapping on $Y$. 
 Then Theorem~\ref{thm:fpt} ensures that there exists $u\in Y$ such that 
 $T_{l+1}u=u$, and hence $u\in \bigcap _{k=1}^{l+1}\Fix(T_k)$. 
 Therefore, the set $\bigcap _{k=1}^{l+1}\Fix(T_k)$ is nonempty. 
\end{proof}

Using Lemma~\ref{lem:fpt}, we next obtain 
the following demiclosed principle for metrically nonspreading mappings. 

\begin{theorem}\label{thm:demiclosed}
 Let $X$ be a metric space, 
 $T$ a metrically nonspreading mapping of $X$ into itself,  
 and $\{x_n\}$ a sequence in $X$ such that 
 $\AC\bigl(\{x_n\}\bigr)=\{p\}$ 
 for some $p\in X$ 
 and $d(Tx_n, x_n)\to 0$ as $n\to \infty$. 
 Then $p$ is an element of $\Fix(T)$. 
\end{theorem}

\begin{proof}
 Since $d(Tx_n, x_n) \to 0$ as $n\to \infty$, 
 we have 
 \[
  \limsup_n d(x_n, y) = \limsup_n d(Tx_n, y)
 \]
 for all $y\in X$. Thus it follows from Lemma~\ref{lem:fpt} 
 that $p$ is an element of $\Fix(T)$. 
\end{proof}

Using Lemma~\ref{lem:double-sequence}, 
we next show the asymptotic regularity of metrically nonspreading mappings. 

\begin{lemma}\label{lem:asymp}
 Let $X$ be a metric space and 
 $T$ a metrically nonspreading mapping of $X$ into itself  
 such that $\Fix (T)$ is nonempty. 
 Then $T$ is asymptotically regular. 
\end{lemma}

\begin{proof}
 Let $x\in X$ be given. 
 Since $\Fix(T)$ is nonempty and $T$ is quasinonexpansive, 
 $\{T^nx\}$ is bounded. 
 Let $A$ be the bounded function of $\N \times \N$ 
 into $[0,\infty)$ defined by 
  $A(n, m) = d(T^nx, T^mx)^2$ 
 for all $n,m\in \N$. 
 It is clear that $A(n,n)=0$ for all $n\in \N$. 
 Since $T$ is metrically nonspreading, we have 
 \[
  2d(T^{n+1}x, T^{m+1}x)^2 
  \leq d(T^{n+1}x, T^{m}x)^2 + d(T^{m+1}x, T^{n}x)^2
 \]
 and hence 
 \[
  2A(n+1,m+1) \leq A(n+1,m) + A(n,m+1)
 \]
 for all $n,m\in \N$. 
 It then follows from Lemma~\ref{lem:double-sequence} that 
 \[
  d(T^{n}x, T^{n+1}x) = \sqrt{A(n,n+1)} \to 0 
 \]
 as $n\to \infty$. Therefore, the mapping $T$ is asymptotically
 regular. 
\end{proof}

We can directly show the asymptotic regularity 
for firmly metrically nonspreading mappings as follows.  

\begin{lemma}\label{lem:asymp_firm}
 Let $X$ be a metric space 
 and $T$ a firmly metrically nonspreading mapping 
 of $X$ into itself  
 such that $\Fix(T)$ is nonempty. 
 Then $T$ is asymptotically regular. 
\end{lemma}
 
\begin{proof}
 Let $x\in X$ be given. By assumption, there exists $u\in \Fix(T)$. 
 Since $T$ is firmly metrically nonspreading, 
 it follows from~\eqref{eq:fmn-quasi} that 
 \begin{align*}
  d(u, T^{n+1}x)^2
  \leq d(u, T^{n+1}x)^2 + d(T^{n+1}x,T^{n}x)^2 
  \leq d(u, T^{n}x)^2. 
 \end{align*}
 This implies that $\{d(u, T^{n}x)^2\}$ is convergent and hence 
 \[
  0\leq d(T^{n+1}x, T^{n}x)^2 
   \leq d(u, T^{n}x)^2 - d(u, T^{n+1}x)^2 
   \to 0
 \]
 as $n\to \infty$. Consequently, 
 the mapping $T$ is asymptotically regular. 
\end{proof}

\begin{theorem}\label{thm:conv}
 Let $X$ be an Hadamard space 
 and $T$ a metrically nonspreading mapping 
 of $X$ into itself 
 such that $\Fix(T)$ is nonempty. 
 Then $\{T^nx\}$ is $\Delta$-convergent to 
 an element of $\Fix(T)$ for all $x\in X$. 
\end{theorem}

\begin{proof}
 Let $z$ be an element of $\omega_{\Delta}\bigl(\{T^nx\}\bigr)$. 
 Then we have a subsequence $\{T^{n_i}x\}$ of $\{T^{n}x\}$ 
 which is $\Delta$-convergent to $z$. 
 In particular, we have $\AC\bigl(\{T^{n_i}x\}\bigr)=\{z\}$. 
 Since $\Fix(T)$ is nonempty, it follows from Lemma~\ref{lem:asymp} that 
 \[
  d\bigl(T(T^{n_i}x), T^{n_i}x\bigr) 
  = d(T^{n_i+1}x, T^{n_i}x) 
  \to 0 
 \]
 as $i\to \infty$. 
 Lemma~\ref{thm:demiclosed} then ensures that $z$ is an 
 element of $\Fix(T)$. 
 It also follows from $d(z, T^{n+1}x) \leq d(z, T^{n}x)$ 
 that $\{d(z, T^{n}x)\}$ is convergent. 
 Thus the sequence $\{d(z, T^{n}x)\}$ is convergent 
 for each $z$ in $\omega_{\Delta}\bigl(\{T^nx\}\bigr)$. 
 By Lemma~\ref{lem:KK-Delta}, we conclude that 
 $\{T^nx\}$ is $\Delta$-convergent to some $u\in X$. 
 Since 
 \[
  \{u\} = \omega_{\Delta}\bigl(\{T^nx\}\bigr) \subset \Fix(T), 
 \]
 we conclude that $u$ is an element of $\Fix(T)$. 
\end{proof}

\section{Asymptotic behavior of the Mann iteration}
\label{sec:Mann}

In this section, we study the asymptotic behavior of 
sequences generated by the Mann iteration~\cite{MR0054846} 
for metrically nonspreading mappings in Hadamard spaces. 

Motivated by~\cite{MR3289781}*{Lemma~3.1}, 
we first show the following equivalence. 

\begin{lemma}\label{lem:equiv-nonsp}
 Let $X$ be a metric space and $T$ a mapping of $X$ into itself.  
 Then $T$ is metrically nonspreading if and only if 
\[
 0\leq d(Ty, y)^2 
 + 2 \ip{\overrightarrow{(Tx)(Ty)}}{\overrightarrow{(Ty)y}} 
 + d(Ty, x)^2 - d(Tx, Ty)^2 
\]
 for all $x,y\in X$. 
 \end{lemma}

\begin{proof}
 Let $x,y\in X$ be given. Then we have 
  \begin{align*}
  &d(Tx, y)^2 + d(Ty, x)^2 - 2d(Tx, Ty)^2 \\
  & = d(Tx, Ty)^2 + d(Ty, y)^2 
 + 2 \ip{\overrightarrow{(Tx)(Ty)}}{\overrightarrow{(Ty)y}} 
 + d(Ty, x)^2 - 2d(Tx, Ty)^2 \\
  & = d(Ty, y)^2 
 + 2 \ip{\overrightarrow{(Tx)(Ty)}}{\overrightarrow{(Ty)y}} 
 + d(Ty, x)^2 - d(Tx, Ty)^2 
 \end{align*}
 and hence the result follows. 
\end{proof}

Motivated by~\cite{MR3289781}*{Lemma~3.2}, 
we next show that every metrically nonspreading mapping is bounded 
on every bounded subset.  

\begin{lemma}\label{lem:bdd_on_bdd}
 Let $X$ be a metric space 
 and $T$ a metrically nonspreading mapping 
 of $X$ into itself.  
 Then $T(U)$ is bounded 
 for each nonempty bounded subset $U$ of $X$.  
\end{lemma}

\begin{proof}
 If the conclusion does not hold, then 
 there exists a bounded sequence $\{x_n\}$ 
 such that $\{Tx_n\}$ is unbounded. Fix $p\in X$. 
 Since 
 \[
  d(Tx_k, Tx_l) \leq d(Tx_k, p) + d(p, Tx_l) \leq 2\sup_{n} d(Tx_n, p)
 \] 
 for all $k, l\in \N$ and $\{Tx_n\}$ is unbounded, 
 we then have $\sup_{n} d(Tx_n, p) =\infty$.  
 This implies that there exists a subsequence $\{Tx_{n_i}\}$ of $\{Tx_n\}$ such that 
 $\{d(Tx_{n_i}, p)\}$ is divergent to $\infty$ as $i\to \infty$.     
 This gives us that 
 \[
  \lim_{i\to \infty} d(Tx_{n_i}, z) = \infty 
 \]
 for all $z\in X$. 
 Let $y\in X$ be given. It follows from Lemma~\ref{lem:equiv-nonsp} that 
\begin{align*}
 &d(Tx_n, Ty)^2 \\
 &\leq d(Ty, y)^2 
 + 2 \ip{\overrightarrow{(Tx_n)(Ty)}}{\overrightarrow{(Ty)y}} 
 + d(Ty, x_n)^2 \\
 &= d(Ty, y)^2 
 + \left(d(Tx_n, y)^2 -d(Tx_n, Ty)^2 - d(Ty, y)^2\right)
 + d(Ty, x_n)^2 \\
 &\leq \bigl(d(Tx_n, y) + d(Tx_n, Ty)\bigr) d(y, Ty) 
 + d(Ty, x_n)^2 \\
 &\leq \bigl(2d(Tx_n, Ty) + d(Ty, y)\bigr) d(y, Ty) 
 + d(Ty, x_n)^2. 
\end{align*}
 Letting $i\to \infty$ in 
 \begin{align*}
  d(Tx_{n_i}, Ty) 
 \leq \left(2+\frac{d(Ty, y)}{d(Tx_{n_i}, Ty)}\right) d(y, Ty) 
 + \frac{d(x_{n_i}, Ty)^2}{d(Tx_{n_i}, Ty)} 
 \end{align*}
 gives us a contradiction. Thus the set $T(U)$ is bounded. 
\end{proof}

We finally show the following result 
on the Mann iteration for metrically nonspreading mappings. 

\begin{theorem}\label{thm:Mann}
 Let $X$ be an Hadamard space, 
 $T$ a metrically nonspreading mapping of $X$ into itself,  
 $\{\alpha_n\}$ a sequence in $(0,1]$ 
 such that $\sum_{n=1}^{\infty}\alpha_n=\infty$, 
 and $\{x_n\}$ a sequence defined by 
 $x_1\in X$ and 
 \[
  x_{n+1} = (1-\alpha_n) x_n \oplus \alpha _n Tx_n 
 \]
 for all $n\in \N$. 
 Then the following hold. 
 \begin{enumerate}
  \item[(i)] $\Fix(T)$ is nonempty if and only if $\{x_n\}$ is bounded; 
  \item[(ii)] if $\Fix(T)$ is nonempty and $\inf_n\alpha_n(1-\alpha_n)>0$, 
then $\{x_n\}$ is $\Delta$-convergent to an element of $\Fix(T)$.  
 \end{enumerate}
\end{theorem}

\begin{proof}
 We first show the only if part of~(i). 
 Suppose that $\Fix(T)$ is nonempty 
 and fix $u\in \Fix(T)$. 
 Since $T$ is quasinonexpansive, we have 
 \begin{align*}
  \begin{split}
  d(u, x_{n+1}) 
 &= d\bigl(u, (1-\alpha_n) x_{n}\oplus \alpha_n Tx_n\bigr) \\
 &\leq (1-\alpha_n) d(u, x_{n}) + \alpha_nd(u, Tx_n) \leq d(u, x_{n})    
  \end{split}
 \end{align*}
 and hence $\{d(u, x_n)\}$ is convergent. 
 This implies that $\{x_n\}$ is bounded. 

 We next show the if part of~(i). 
 Suppose that $\{x_n\}$ is bounded 
 and set 
 \[
  \omega_n = \sum_{l=1}^{n} \alpha_l
 \]
 for all $n\in \N$. 
 Then it follows from Lemma~\ref{lem:bdd_on_bdd} 
 that $\{Tx_n\}$ is bounded. 
 Let $g$ be the real function on $X$ defined by 
 \[
  g(y) = \limsup_{n} 
 \frac{1}{\omega_n} 
 \sum_{k=1}^{n} \alpha_k d(y, Tx_k)^2 
 \]
 for all $y\in X$. 
 It then follows from Theorem~\ref{thm:unique_minimizer} that 
 $g$ has a unique minimizer $p\in X$. 
 By the definition of $\{x_n\}$, we have 
 \begin{align}\label{eq:thm:Mann-a}
 \begin{split}
 d(Tp, x_{k+1})^2 
 &= d\bigl(Tp, (1-\alpha_k) x_k \oplus \alpha_k Tx_k\bigr)^2  \\  
 &\leq (1-\alpha_k) d(Tp, x_k)^2 + \alpha_k d(Tp, Tx_k)^2  
 \end{split}
 \end{align}
 Since $T$ is metrically nonspreading, we have 
 \begin{align}\label{eq:thm:Mann-b}
  2d(Tp, Tx_k)^2 \leq d(Tp, x_k)^2 + d(Tx_k, p)^2. 
 \end{align}
 Using~\eqref{eq:thm:Mann-a} and~\eqref{eq:thm:Mann-b}, we have 
 \begin{align*}
 &\alpha_kd(Tp, Tx_k)^2 \\
 &\leq \alpha_k d(Tx_k, p)^2 
 + \alpha_k \bigl(d(Tp, x_k)^2 - d(Tp, Tx_k)^2\bigr) \\
 &= \alpha_k d(Tx_k, p)^2 
 + d(Tp, x_k)^2 -\bigl((1-\alpha_k)d(Tp, x_k)^2 + \alpha_k d(Tp,
  Tx_k)^2\bigr) \\
 &\leq \alpha_k d(Tx_k, p)^2 
 + d(Tp, x_k)^2 -d(Tp, x_{k+1})^2. 
 \end{align*}
 Consequently, we have 
 \begin{align*}
  \sum_{k=1}^{n} \alpha_k d(Tp, Tx_k)^2 
  & \leq \sum_{k=1}^{n} \alpha_k d(p, Tx_k)^2 
   +\bigl(d(Tp, x_1)^2 -d(Tp, x_{n+1})^2\bigr) \\
  & \leq \sum_{k=1}^{n} \alpha_k d(p, Tx_k)^2 
   +d(Tp, x_1)^2
 \end{align*}
 and hence 
 \begin{align*}
  \frac{1}{\omega_n}
 \sum_{k=1}^{n} \alpha_k d(Tp, Tx_k)^2 
   \leq \frac{1}{\omega_n}
 \sum_{k=1}^{n} \alpha_k d(p, Tx_k)^2 
  + \frac{1}{\omega_n} d(Tp, x_1)^2. 
 \end{align*}
 Since $\omega_n\to \infty$ as $n\to \infty$, 
 we obtain $g(Tp) \leq g(p)$. 
 Since $p$ is the unique minimizer of $g$, we conclude that $Tp=p$. 

 We finally show~(ii). 
 Suppose that $\Fix(T)$ is nonempty 
 and $\inf_n \alpha_n(1-\alpha_n) > 0$. 
 Fix $v\in \Fix(T)$. 
 Since $X$ is a $\CAT(0)$ space, we have 
 \begin{align*}
 \begin{split}
  d(v, x_{n+1})^2 
 &= d\bigl(v, (1-\alpha_n) x_n \oplus \alpha_n Tx_n\bigr)^2 \\
 &\leq (1-\alpha_n) d(v, x_n)^2 + \alpha_n d(v, Tx_n)^2 
 -\alpha_n (1-\alpha_n) d(x_n, Tx_n)^2 \\
 &\leq d(v, x_n)^2 
 -\alpha_n (1-\alpha_n) d(x_n, Tx_n)^2.   
 \end{split}
 \end{align*}
 This gives us that $\{d(v, x_{n})^2\}$ is convergent 
 and hence 
 \[
  d(x_n, Tx_n)^2 \leq \frac{1}{\inf_m \alpha_m (1-\alpha_m)} 
 \left(d(v, x_n)^2 - d(v, x_{n+1})^2\right) \to 0 
 \]
 as $n\to \infty$. Consequently, we obtain $d(x_n, Tx_n)\to 0$ 
 as $n\to \infty$. 
 If $z$ is an element of $\omega_{\Delta}\bigl(\{x_n\}\bigr)$, 
 then there exists a subsequence $\{x_{n_i}\}$ of $\{x_n\}$ 
 which is $\Delta$-convergent to $z$. 
 Since 
 \[
  \lim_{i}d(x_{n_i}, Tx_{n_i})= 0, 
 \]
 Lemma~\ref{thm:demiclosed} ensures that 
 $z$ is an element of $\Fix(T)$. 
 Hence $\omega_{\Delta}\bigl(\{x_n\}\bigr)$ 
 is a subset of $\Fix(T)$. 
 Thus, the sequence $\{d(z, x_n)\}$ is convergent 
 for each $z$ in $\omega_{\Delta}\bigl(\{x_n\}\bigr)$. 
 Then, Lemma~\ref{lem:KK-Delta} implies that 
 $\{x_n\}$ is $\Delta$-convergent to some $x_{\infty}\in X$. 
 Since 
 \[
  \{x_{\infty}\} 
  =\omega_{\Delta}\bigl(\{x_n\}\bigr) 
  \subset \Fix(T), 
 \]
 we conclude that $x_{\infty}$ is an element of $\Fix(T)$. 
\end{proof}

\section{Applications to monotone operators in Hadamard spaces}
\label{sec:app}

In this section, we obtain two corollaries of our results 
for the problem of finding zero points of monotone operators in Hadamard spaces. 

Before obtaining them, we first summarize the concepts of dual spaces and 
monotone operators in $\CAT(0)$ spaces. 
These concepts were introduced 
by Ahmadi Kakavandi and Amini~\cite{MR2680038}; 
see also Ahmadi Kakavandi~\cite{MR3003694} on related results. 

Let $X$ be a $\CAT(0)$ space and $\hat{L}(X)$ 
the real linear space of all Lipschitz continuous real functions on $X$. 
We denote by $\norm{\,\cdot \,}$ the Lipschitz seminorm on 
$\hat{L}(X)$ defined by 
\[
 \norm{f} = \sup\left\{\frac{\abs{f(p)-f(q)}}{d(p, q)}: p, q\in X, \,
 p\neq q\right\}
\]
for all $f\in \hat{L}(X)$. In other words, 
\[
 \norm{f} = \min \bigl\{\lambda \in [0,\infty): 
 \abs{f(p)-f(q)}\leq \lambda d(p,q) 
 \quad (\forall p, q\in X)\bigr\}
\]
 for all $f\in \hat{L}(X)$. 
 Then the following conditions 
 \[
  \norm{f}\geq 0; \quad 
  \norm{\alpha f} = \abs{\alpha} \norm{f}; \quad  
  \norm{f+g}\leq \norm{f} + \norm{g}
 \]
 hold for all $f,g\in \hat{L}(X)$ and $\alpha \in \R$. 
 
 We can define an equivalence relation $\sim$ on $\hat{L}(X)$ by 
 $f\sim g$ if $\norm{f-g}=0$. 
 It is clear that $f\sim g$ if and only if $f-g$ is a constant function. 
 We denote by $[f]$ the equivalence class of $f\in \hat{L}(X)$ 
 and let $L(X)$ be the space defined by 
 \[
  L(X)=\{[f]: f\in \hat{L}(X)\}. 
 \]
 The space $L(X)$ is a real Banach space under 
 the addition, the scalar multiplication, and the norm given by 
 \[
  [f] + [g] = [f+g], \quad \alpha [f] = [\alpha f], \quad \norm{[f]} = \norm{f}
 \]
 for all $[f], [g]\in L(X)$ and $\alpha \in \R$; 
 see~\cite{MR1442257}*{Proposition~2.4.1} for the proof of the metric completeness of $L(X)$. 

 We denote by $\alpha \overrightarrow{xy}$ 
 and $\overrightarrow{xy}$ the elements 
 $(\alpha, \overrightarrow{xy})$ and 
 $(1, \overrightarrow{xy})$ in $\R\times X^2$, respectively. 
 Then we define the mapping $\Phi$ 
 of $\R\times X^2$ into $\hat{L}(X)$ by 
 \[
  \Phi (\alpha \overrightarrow{xy})(p)
  = \alpha \ip{\overrightarrow{xy}}{\overrightarrow{xp}}
 \]
 for all $\alpha\overrightarrow{xy}\in \R\times X^2$ 
 and $p\in X$. It is easy to see that 
 \[
  \norm{\Phi(\alpha\overrightarrow{xy})} = \abs{\alpha}d(x,y). 
 \]
 We also define a real function $\hat{D}$ by  
 \begin{align}
 \begin{split}
  \hat{D}\bigl(\alpha \overrightarrow{xy}, \beta
  \overrightarrow{zw}\bigr) 
 &= \norm{\Phi(\alpha \overrightarrow{xy}) - 
 \Phi (\beta \overrightarrow{zw})} 
 \end{split}
 \end{align}
 for all 
 $\alpha \overrightarrow{xy}, \beta \overrightarrow{zw} \in \R\times X^2$, 
 which is a pseudometric on $\R \times X^2$, that is, 
 \begin{itemize}
  \item $\hat{D}(\alpha \overrightarrow{xy}, 
 \beta \overrightarrow{zw})\geq 0$ and $\hat{D}(\alpha \overrightarrow{xy}, 
 \alpha \overrightarrow{xy}) = 0$; 
  \item $\hat{D}(\alpha \overrightarrow{xy}, 
 \beta \overrightarrow{zw})
 =\hat{D}(\beta \overrightarrow{zw}, 
\alpha \overrightarrow{xy})$; 
  \item $\hat{D}(\alpha \overrightarrow{xy}, 
 \beta \overrightarrow{zw}) 
  \leq \hat{D}(\alpha \overrightarrow{xy}, 
 \gamma \overrightarrow{uv}) 
 + \hat{D}(\gamma \overrightarrow{uv}, 
 \beta \overrightarrow{zw})$
 \end{itemize}
 hold for all $\alpha \overrightarrow{xy}, \beta\overrightarrow{zw}, 
 \gamma \overrightarrow{uv} 
 \in \R \times X^2$. 
 We can define an equivalence relation $\sim$ on $\R \times X^2$ by 
 \[
  \alpha \overrightarrow{xy} \sim 
  \beta \overrightarrow{zw} 
  \Longleftrightarrow 
  \hat{D}(\alpha \overrightarrow{xy}, 
 \beta \overrightarrow{zw})= 0. 
 \]
 We then have the following equivalence;
 see also~\cite{MR2680038}*{Lemma~2.1}. 
 \begin{align}
 \begin{split}\label{eq:charact-equiv}
 &\alpha \overrightarrow{xy} \sim 
  \beta \overrightarrow{zw} \\
 &\Longleftrightarrow \Phi(\alpha \overrightarrow{xy}) \sim 
  \Phi(\beta \overrightarrow{zw}) \\
 &\Longleftrightarrow 
 \Phi(\alpha \overrightarrow{xy})(p) - \Phi(\beta 
  \overrightarrow{zw})(p) = 
 \Phi(\alpha \overrightarrow{xy})(q) - \Phi(\beta 
  \overrightarrow{zw})(q) 
\quad (\forall p, q\in X) \\
 &\Longleftrightarrow \alpha
  \ip{\overrightarrow{xy}}{\overrightarrow{xp}} 
-\alpha
  \ip{\overrightarrow{xy}}{\overrightarrow{xq}} 
= 
 \beta \ip{\overrightarrow{zw}}{\overrightarrow{zp}} 
-\beta \ip{\overrightarrow{zw}}{\overrightarrow{zq}} 
\quad (\forall p, q\in X) \\
 &\Longleftrightarrow \alpha \ip{\overrightarrow{xy}}{\overrightarrow{pq}} = 
 \beta \ip{\overrightarrow{zw}}{\overrightarrow{pq}} 
\quad (\forall \overrightarrow{pq} \in X^2). 
 \end{split}
 \end{align}
 The dual space $X^*$ of $X$ in the sense of 
 Ahmadi Kakavandi and Amini~\cite{MR2680038} 
 is the metric space given by 
 \[
  X^* = \{[\alpha \overrightarrow{xy}]: \alpha \overrightarrow{xy}
 \in \R\times X^2\}
 \]
 with the metric $D$ defined by 
 \[
  D\bigl([\alpha \overrightarrow{xy}], [\beta
  \overrightarrow{zw}]\bigr)
 = \hat{D}(\alpha \overrightarrow{xy}, \beta 
  \overrightarrow{zw}\bigr) 
 \]
 $[\alpha \overrightarrow{xy}], [\beta\overrightarrow{zw}] \in X^*$, 
 where $[\alpha \overrightarrow{xy}]$ denotes the equivalence class 
 of $\alpha \overrightarrow{xy} \in \R \times X^2$. 
 The origin $0$ of $X^*$ is given by 
 \[
  0=[\overrightarrow{aa}], 
 \] 
 where $a$ is a fixed element of $X$. It is obvious that 
 \[
  0= \{0 \overrightarrow{xy}: x, y\in X\} 
  \cup \{\alpha \overrightarrow{xx}: \alpha \in \R,\, x\in X\}. 
 \]

 It is known~\cite{MR2680038}*{pp.~3451--3452} that 
 if $X$ is a closed convex subset of a real Hilbert space $H$ 
 with a nonempty interior, 
 then $X$ is isometric to $X^*$. 
 In particular, if $X=H$, then the isometric bijection 
 $\tau$ of $H$ onto $H^*$ is given by 
 \[
  \tau (x) = [(1,0,x)]
 \]
 for all $x\in H$. 
 For each $u^*=[\alpha \overrightarrow{xy}]\in X^*$,  
 we define 
 \[
  \ip{u^*}{\overrightarrow{pq}}
  = \alpha \ip{\overrightarrow{xy}}{\overrightarrow{pq}}
 \]
 for all $\overrightarrow{pq}\in X^2$. 
 It follows from~\eqref{eq:charact-equiv} 
 that this is independent of the choice of $\alpha\overrightarrow{xy}$. 
 
 We next recall the concept of monotone operators in $\CAT(0)$ spaces. 
 Let $X$ be a $\CAT(0)$ space and $X^*$ the dual space of $X$. 
 Then an operator $A\colon X\to 2^{X^*}$ is said to be monotone if 
 \[
  \ip{u^*}{\overrightarrow{vu}} - \ip{v^*}{\overrightarrow{vu}} \geq 0
 \]
 whenever $u^*\in Au$ an $v^*\in Av$. 
 A monotone operator $A\colon X\to 2^{X^*}$ is said to 
 satisfy a range condition if 
 for each $x\in X$, 
 there exists $z\in X$ such that 
 \[
  [\overrightarrow{zx}] \in Az. 
 \]
 If $A\colon X\to 2^{X^*}$ is a monotone operator satisfying a range
 condition, then the resolvent $J_A$ of $A$ defined by 
 \[
  J_A (x) = \{z\in X: [\overrightarrow{zx}]\in Az\}
 \]
 for all $x\in X$ is a single-valued mapping of $X$ into itself. 
 The zero point set $A^{-1}(0)$ is defined by 
 \[
  A^{-1}(0) 
  =\bigl\{u\in X: 0 \in Au\bigr\}. 
 \]

 For a proper lower semicontinuous convex function 
 $f$ of a $\CAT(0)$ space $X$ into $(-\infty, \infty]$, 
 the subdifferential mapping $\partial f\colon X\to 2^{X^*}$ 
 of $f$ in the sense of Ahmadi Kakavandi and Amini~\cite{MR2680038}*{Definition~4.1} 
 is defined by 
 \[
  \partial f(x) = 
  \bigl\{u^*\in X^*: f(x) + \ip{u^*}{\overrightarrow{xy}} \leq f(y) 
  \quad (\forall y\in X)\bigr\}
 \]
 for all $x\in X$. 
 It is known~\cite{MR2680038}*{Theorem~4.2} that 
 if $X$ is an Hadamard space, then $\partial f\colon X\to 2^{X^*}$ is a monotone operator 
 satisfying a range condition 
 and 
 \[
  (\partial f)^{-1}(0) 
  = \{u\in X: f(u) = \inf f(X)\}. 
 \]
 It is also known~\cite{MR3679017}*{Proposition~5.3} that 
 the resolvent $J_{\partial f}$ of $\partial f$ 
 coincides with the proximity mapping $\Prox_f$ of $f$ 
 defined by~\eqref{eq:prox}. 

 We know the following fundamental result. 
 \begin{lemma}[\cite{MR3679017}*{Theorem~3.9}]
 \label{lem:mono-res}
 Let $X$ be a $\CAT(0)$ space, $A\colon X\to 2^{X^*}$  
 a monotone operator satisfying a range condition, 
 and $J_A$ the resolvent of $A$. 
 Then $J_A$ is a firmly metrically nonspreading 
 mapping of $X$ into itself such that 
 $\Fix(J_A)=A^{-1}(0)$. 
 \end{lemma}

For the sake of completeness, we give the proof. 

\begin{proof}
 Put $T=J_A$. 
 The definition of $T$ gives us that 
 \[
  u=Tu
  \Longleftrightarrow 
  0=[\overrightarrow{uu}] \in Au 
 \] 
 and hence $\Fix(T)=A^{-1}(0)$. 

 We next show that $T$ is 
 firmly metrically nonspreading. 
 Let $x,y\in X$ be given. By the definition of $T$, 
 we have 
 \[
  \left[\overrightarrow{(Tx)x}\right]\in A(Tx) 
  \quad \textrm{and} \quad 
  \left[\overrightarrow{(Ty)y}\right]\in A(Ty). 
 \]
 The monotonicity of $A$ implies that 
 \[
  \ip{\left[\overrightarrow{(Tx)x}\right]}
  {\overrightarrow{(Ty)(Tx)}}
  -
  \ip{\left[\overrightarrow{(Ty)y}\right]}
  {\overrightarrow{(Ty)(Tx)}}
  \geq 0 
 \]
 and hence 
 \[
  \ip{\overrightarrow{(Tx)x}}
  {\overrightarrow{(Ty)(Tx)}}
  -
  \ip{\overrightarrow{(Ty)y}}
  {\overrightarrow{(Ty)(Tx)}}
  \geq 0.  
 \]
 By the definition of quasilinearization, we have 
 \[
  d(x,Ty)^2 - d(Tx,Ty)^2 - d(x,Tx)^2
  -
  \left(
  d(Ty,Tx)^2 + d(y,Ty)^2 - d(y,Tx)^2
  \right)
  \geq 0
 \]
 and hence 
 \begin{align*}
  d(Tx, y)^2 + d(Ty, x)^2 - d(Tx, x)^2 - d(Ty, y)^2 
  \geq 2d(Tx, Ty)^2. 
 \end{align*}
 Therefore, the mapping $T$ is firmly metrically nonspreading. 
\end{proof}

Using Theorems~\ref{thm:fpt},~\ref{thm:conv}, 
and Lemma~\ref{lem:mono-res}, 
we obtain the following corollary. 
The part~(ii) also follows from a more general result 
in~\cite{MR3679017}*{Theorem~4.3}. 

\begin{corollary}\label{cor:mono-fpt}
 Let $X$ be an Hadamard space, 
 $A\colon X\to 2^{X^*}$ a monotone operator 
 satisfying a range condition, and $J_A$ the resolvent of $A$. 
 Then the following hold. 
 \begin{enumerate}
  \item[(i)] The set $A^{-1}(0)$ is nonempty if and only if 
 $\{(J_A)^nx\}$ is bounded for some $x\in X$; 
  \item[(ii)] if $A^{-1}(0)$ is nonempty, 
 then $\{(J_A)^nx\}$ is $\Delta$-convergent to 
 an element of $A^{-1}(0)$ for all $x\in X$. 
 \end{enumerate}
\end{corollary}

Using Theorem~\ref{thm:Mann} and Lemma~\ref{lem:mono-res}, 
we obtain the following corollary. 

\begin{corollary}\label{cor:mono-Mann}
 Let $X$ be an Hadamard space, 
 $A\colon X\to 2^{X^*}$ a monotone operator 
 satisfying a range condition, 
 $J_A$ the resolvent of $A$, 
 $\{\alpha_n\}$ a sequence in $(0,1]$ 
 such that $\sum_{n=1}^{\infty} \alpha_n =\infty$, 
 and $\{x_n\}$ a sequence defined by 
 $x_1\in X$ and 
 \[
  x_{n+1} = (1-\alpha_n) x_n \oplus \alpha _n J_Ax_n 
 \]
 for all $n\in \N$. 
 Then the following conditions hold. 
 \begin{enumerate}
  \item[(i)] The set $A^{-1}(0)$ is nonempty if and only if $\{x_n\}$ is bounded; 
  \item[(ii)] if $A^{-1}(0)$ is nonempty and 
$\inf_n\alpha_n(1-\alpha_n)>0$, 
then $\{x_n\}$ is $\Delta$-convergent to an element of
	      $A^{-1}(0)$.   
 \end{enumerate}
\end{corollary}

\section*{Acknowledgment}
This work was supported by JSPS KAKENHI Grant Number 17K05372.

\end{document}